\documentclass[12pt]{amsart}

\usepackage{amsmath,enumitem}
\usepackage[english,activeacute]{babel}
\usepackage{mathrsfs} 
\usepackage[latin1]{inputenc}
\usepackage{amssymb}
\usepackage{amsthm}
\usepackage{graphicx}
\usepackage{array}
\usepackage{pdflscape}
\usepackage{a4wide}
\usepackage{url}
\usepackage{float}
\usepackage{hyperref}
\usepackage{multirow}
\usepackage[capitalize,nameinlink,noabbrev,nosort]{cleveref}

\usepackage[dvipsnames]{xcolor}

\hypersetup{
    pdftoolbar=true,        
    pdfmenubar=true,        
    pdffitwindow=false,     
    pdfstartview={FitH},    
    pdftitle={},
    pdfauthor={},
    pdfsubject={},
    pdfkeywords={},
    pdfnewwindow=true,      
    colorlinks=true,       
    linkcolor=brown,          
    citecolor=brown,        
    filecolor=brown,      
    urlcolor=brown,           
    unicode=true          
}

\numberwithin{equation}{section}

\theoremstyle{plain}
\newtheorem{theorem}{Theorem}[section]
\newtheorem{corollary}[theorem]{Corollary}
\newtheorem{lemma}[theorem]{Lemma}
\newtheorem{proposition}[theorem]{Proposition}

\theoremstyle{definition}
\newtheorem{definition}[theorem]{Definition}

\newtheorem{example}[theorem]{Example}
\newtheorem{remark}[theorem]{Remark}

\def\st#1#2{\begin{bmatrix}
#1 \\ #2 \end{bmatrix}}
\def\sts#1#2{\begin{Bmatrix}
#1 \\ #2 \end{Bmatrix}}
\newcommand{\eulerian}[2]{\left\langle \begin{matrix}
#1 \\ #2 \end{matrix} \right\rangle}

\def\piros#1{\underline{#1}}%
\def\kek#1{{\overline{#1}}}%

\newcommand{\ds}{\displaystyle}

\DeclareMathOperator{\Li}{Li}

\newcommand{\floor}[1]{\lfloor #1 \rfloor}
\newcommand{\ceil}[1]{\lceil #1 \rceil}

\title
{Toppling on permutations with an extra chip}

\author{Arvind Ayyer}
\address{Arvind Ayyer, Department of Mathematics, Indian Institute of Science, Bangalore - 560012, India}
\email{arvind@iisc.ac.in}

\author{Be\'ata B\'enyi}
\address{Be\'{a}ta B\'enyi, Faculty of Water Sciences, University of Public Service, Baja, Hungary}
\email{benyi.beata@uni-nke.hu}

\date{\today}
\begin{document}

\begin{abstract}
The study of toppling on permutations with an extra labeled chip was initiated by the first author with D. Hathcock and P. Tetali 
(arXiv:2010.11236), where the extra chip was added in the middle.
We extend this to all possible locations $p$ as well as values $r$ of the extra chip and give a complete characterization of permutations which topple to the identity. 
Further, we classify all permutations which are outcomes of the toppling process in this generality, which we call resultant permutations. Resultant permutations turn out to be certain decomposable permutations. The number of configurations toppling to a given resultant permutation is shown to depend purely on the number of left-to-right maxima (or records) of the permutation to the left of $p$ and the number of right-to-left minima to the right of $p$.
The number of permutations toppling to a given resultant permutation (identity or otherwise) is shown to be the binomial transform of a poly-Bernoulli number of type B.
\end{abstract}


\maketitle


\section{Introduction}


Chip-firing (also called abelian sandpile model) is a stochastic discrete dynamical system defined on a graph. 
Hopkins--McConville--Propp~\cite{HMP} introduced a labeled version of the chip-firing process on the infinite path graph. They showed the remarkable property that when the chips start at the origin and the number of chips is even the chips always end up in sorted order. Many variants of this original problem have been considered since then~\cite{galashin-et-al-2019, Galashin2021, hopkins-postnikov-2019,  klivans-liscio-2020+, klivans-felzenszwalb-2021}.

The first author with D. Hathcock and P. Tetali~\cite{AHT} considered a variant where a permutation of chips labeled $1$ through $n$ were placed contiguously on the path graph and one extra chip, labelled $r$ was added in the middle. In that case, they showed that the final configuration is deterministic. 
They defined a toppleable permutation to be one which results in the identity permutation for all values of $r$, and showed that
the number of toppleable permutations is a poly-Bernoulli number of type C, even though they did not use that language.

In this paper we investigate this model in much greater generality and 
explain the connection to the combinatorics of the poly-Bernoulli numbers. 
Poly-Bernoulli numbers of types B and C arise naturally in many combinatorial problems, such as permutations according to excedance sets, acyclic orientations of bipartite graphs, lonesum matrices, Vesztergombi permutations and a lot more; see \cite{benyi-hajnal-2017} for a review.

Our first results are a natural generalization of the formulation of \cite{AHT}. We add the extra chip $r$ to an arbitrary site $p$ in the permutation.
We consider toppleability in two different but related ways. In the first, we fix $p$ and count all possible configurations that result in a sorted configuration. In the second, we count permutations which are toppleable for fixed $r$ and $p$.
In both cases, we show that the numbers are related to the poly-Bernoulli numbers of type B. 
Generalizing \cite[Theorem 2.4]{AHT}, we also count permutations which are toppleable for all $r$ for a fixed $p$.
In this case, we show that we always obtain poly-Bernoulli numbers of type C. 

Our second results give a characterization of all the permutations that can arise as the result of a toppling procedure, which we call resultant permutations. 
It turns out that resultant permutations can be succintly characterized by left-to-right maxima to the left of $p$ and right-to-left minima to the right of $p$.
We also enumerate toppling to these resultant permutations.
Here too we study configurations with a fixed $p$ and permutations with a fixed $r$ and $p$. In both cases, we show that the enumeration is related to the poly-Bernoulli numbers of type B. 

The plan of the rest of this article is as follows.
In \cref{sec:model}, we introduce the model, give the necessary background and state the main results. 
We give a self-contained summary of combinatorial aspects of poly-Bernoulli numbers of types B and C in \cref{sec:poly-bernoulli}.
\cref{sec:toppleable} is devoted to the understanding of configurations and permutations which topple to the sorted configuration.
Finally, we classify all resultant permutations and enumerate the number of configurations and permutations that topple to them in \cref{sec:finalconfigs}.

\section{The toppling model}
\label{sec:model}

Let $L_n$ be the line segment $L_n=\{0,1,\ldots,n+1\}$. 
We distribute $n+1$ chips labeled by $\{1,2,\ldots, n+1\}$ on $L_n$ as follows: we first distribute $n$ of these chips on the sites $1$ 
through $n$, and then we add the remaining chip on a site $p$, $1 \leq p \leq n$.
The set of all such configurations is denoted $\mathcal{S}(n,p)$.

We define a dynamical system on $L_n$ by the process of \emph{toppling}, which is defined as follows:

\begin{itemize}
	\item[1.] If no position in $L_n$ has two or more chips stop. Else, go to step 2.
	\item[2.] Choose a position $i$ uniformly at random among positions occupied by more than one chip.
	\item[3.] Pick two chips $\alpha<\beta$ uniformly from those at site $i$.
	\item[4.] Move $\alpha$ to position $i-1$ and $\beta$ to $i+1$.
	\item[5.] Go to step 1. 
\end{itemize}

For instance, the configuration denoted by $C=(7,3,1,5,(2,4),6,8) \in \mathcal{S}(7,5)$ has two chips on the fifth site and is depicted as
\begin{equation}
\label{eg-config}
\begin{array}{ccccccccc}
	&&&&&4&&&\\
	&7&3&1&5&2&6&8\\\hline
	0&1&2&3&4&5&6&7&8
\end{array}.
\end{equation}
Let $S_n$ be the set of permutations of $[n]=\{1,2,\ldots, n\}$.
Configurations in $\mathcal{S}(n,p)$ naturally arise from permutations in $S_n$. Given a permutation $\pi=\pi_1\pi_2\ldots\pi_n$ in one-line notation and an element $r\in [n+1]$, we first place the chips labeled by $\pi_i$ on site $i$ for $1 \leq i \leq n$. We then place the chip labeled by $r$ on site $p$ and increase each value $\pi_i\geq r$ by $1$. 
It is easy to see that every configuration can arise in two ways from a permutation in the above described way by the choice of $r$ at site $p$. 
For example, the configuration in \eqref{eg-config} arises from the permutation $\pi=6214357$ with $r=2$ and from the permutation $\sigma = 6314257$ with $r=4$.
An initial configuration $C \in \mathcal{S}(n,p)$ that arises from a permutation $\pi \in S_n$ and $r \in [n+1]$ is denoted $\pi^{(r,p)}$. 
The following result follows from the proof of \cite[Proposition 2.1]{AHT} with almost no change.

\begin{proposition}
\label{prop:determ}
\begin{enumerate}
\item At every step the configuration lives in $L_n$. No chips moves to the left of the site $0$ or to the right of the site $n+1$.
\item The final configuration is deterministic containing exactly one chip on every site except one. 
\end{enumerate}
\end{proposition}

Another important and useful property is the symmetry property.
This is easy to prove by analyzing what happens at a single toppling step.

\begin{proposition}[Symmetry]
	\label{prop:symm}
	For a positive integer $n$, fix $p \leq n-1$. The toppling dynamics on a configuration $C \in \mathcal{S}(n,p)$ is isomorphic to that of $\hat{C} \in \mathcal{S}(n,n-p)$, where $\hat{C}$ is obtained by reversing the direction and subtracting each chip value in $C$ from $n+1$.
\end{proposition}

\begin{definition}
\label{def:topp-config}
We say that a configuration $C \in \mathcal{S}(n,p)$ is  \emph{$p$-toppleable} if the final configuration is sorted.
\end{definition}

Our first main result is for the number of $p$-toppleable configurations. To state our results, we recall the relevant sequences. 
The well-known \emph{polylogarithm} function is given by
\[
\Li_k(z) =\sum_{i=1}^{\infty}\frac{z^i}{i^k}.
\]
When the base $k$ is a non-positive integer, then it is well-known to be a rational function. In particular, for a non-negative integer $m$,
\[
\Li_{-m}(z) = \frac{\ds \sum_{j=0}^{m-1} \eulerian mj z^{m-j} }{(1-z)^{m+1}},
\]
where $\eulerian mj$ is the number of permutations in $S_n$ with $j$ ascents. (Recall that a position $k$ is an ascent in a permutation if $\pi_k < \pi_{k+1}$.)
\emph{Poly-Bernoulli numbers of type B} are defined by the generating function, 
\begin{align}
\label{bnk-gf}
\sum_{n=0}^{\infty}B_{n,k}\frac{x^n}{n!} =\frac{\Li_{-k}(1-e^{-x})}{1-e^{-x}},
\end{align} 
where $k \geq 0$. The reason for this terminology is that $\Li_1(x) = -\log(1-x)$ and so, 
\[
\frac{\Li_{1}(1-e^{-x})}{1-e^{-x}} = \frac{x}{1-e^{-x}},
\]
the generating function of the \emph{Bernoulli numbers}.
The related family of 
\emph{poly-Bernoulli numbers of type C} are defined by the generating function, 
\begin{align}
\label{cnk-gf}
\sum_{n=0}^{\infty}C_{n,k} \frac{x^n}{n!} = \frac{\Li_{-k}(1-e^{-x})}{e^x-1}.
\end{align}
It is also the value of the Arakawa-Kaneko function \cite{ArakawaKaneko}
\[
\xi_k(-n)= (-1)^n C_{n,k},
\]
where the Arakawa-Kaneko function is defined as
\[
\xi_k(s):= \frac{1}{\Gamma(s)}\int_0^{\infty}\frac{t^{s-1}}{e^t-1}\Li_k(1-e^{-t})dt.
\] 
See \cref{tab:data-polybern}(a) and (b) for the first few poly-Bernoulli numbers of types B and C respectively. 
In particular, notice the nontrivial symmetries, $B_{n,k} = B_{k,n}$ and $C_{n+1,k} = C_{k+1,n}$.
\cref{sec:poly-bernoulli} contains more details about these numbers and their relation to combinatorics.

\begin{table}[h!]
	\begin{center}
	\begin{tabular}{cc}
		\begin{tabular}{|c||c|c|c|c|c|c|}
			\hline
			$n \backslash k$
			& 0 & 1 & 2 & 3 & 4 & 5\\
			\hline\hline
			0 & 1 & 1& 1& 1 & 1 & 1\\
			\hline
			1 & 1 & 2 & 4 & 8 & 16 & 32\\
			\hline
			2 & 1 & 4 & 14 & 46 & 146 & 454 \\
			\hline
			3 & 1 & 8 & 46 & 230 & 1066 & 4718 \\
			\hline
			4 & 1 & 16 & 146 & 1066 & 6906 & 41506\\
			\hline
			5 & 1 & 32 & 454 & 4718 & 41506 & 329462\\
			\hline
		\end{tabular}
	&
	\begin{tabular}{|c||c|c|c|c|c|c|}
		\hline
		$n \backslash k$ & 0 & 1 & 2 & 3 & 4 & 5 \\
		\hline\hline
		0  & 1 & 0 & 0 & 0 & 0 & 0\\
		\hline
		1  & 1 & 1 & 1& 1& 1 & 1\\
		\hline
		2   & 1 & 3 & 7 & 15 & 31 & 63 \\
		\hline
		3  & 1 & 7 & 31 & 115 & 391 & 1267 \\
		\hline
		4 & 1 & 15 & 115 & 675 & 3451 & 16275\\
		\hline
		5  & 1 & 31 & 391 & 3451 & 25231 & 164731\\
		\hline
	\end{tabular}\\
	(a) & (b)
	\end{tabular}
	\end{center}
	\caption{The poly-Bernoulli numbers $B_{n,k}$ in (a) and $C_{n,k}$ in (b) for $0 \leq n,k \leq 5$.}
\label{tab:data-polybern}
\end{table}

\begin{theorem}\label{thm: top-pB}
The number of $p$-toppleable configurations in $\mathcal{S}(n,p)$ is given by $B_{n-p+1,p}/2$.
\end{theorem}

Although it is not obvious from any of the definitions, it turns out that $B_{n,k}$ is even if $n,k > 0$; see \cref{rem:bnk-even}. As a consequence, \cref{thm: top-pB} makes sense.
We also formulate an analogous definition to \cref{def:topp-config} in terms of the corresponding permutations. 

\begin{definition}
\label{def:topp-perm}
	We say that a permutation $\pi \in S_n$ is \emph{$(r,p)$-toppleable} if starting with the configuration $\pi^{(r,p)}$ the final configuration of the toppling process is sorted. 
\end{definition}

We denote the set of $(r,p)$-toppleable permutations of $[n]$ by $\mathcal{T}_n^{(r,p)}$.

Let $\Delta$ be the discrete (forward) difference operator,
i.e. for any function $f(n)$, $\Delta(f(n)) = f(n+1) - f(n)$. Then the higher difference operators are obtained by composition. For example, $\Delta^2 (f(n)) = f(n+2) - 2 f(n+1) + f(n)$. Note that $\Delta^0 (f(n)) = f(n)$. 
Given a sequence $(a_n)$, it's \emph{binomial transform}~\cite[Section 5.2.2, Exercise 36]{knuth-taocp3} is given by the sequence $(b_n)$ where
\[
b_n = \sum_{k=0}^{n}(-1)^{k}\binom{n}{k} a_k.
\]
It is easy to see that $\Delta^{m}f(n)$ is $(-1)^m$ times the $(n+m)$'th term of the binomial transform of the sequence $(f(n))$.
Our second main result is on the enumeration of $(r,p)$-toppleable permutations.

\begin{theorem}
\label{thm:r-p-top-perm}
Let $n,p,r$ be integers satisfying $1\leq p\leq n$, $1 \leq r\leq n-p+1$. Then 
$|\mathcal{T}_n^{(r,p)}| = \Delta^{r-1} \big( B_{n-p+1-r,p} \big)$, where $\Delta$ acts on the first index.
\end{theorem}

We then recover one of the main results of \cite{AHT} as a special case.

\begin{corollary}[{\cite[Theorem 3.4 and Lemma 3.5]{AHT}}]
\label{cor:aht-theorem}
	Let $n$ be an odd integer and $p=\floor{(n+1)/2}$. Then for $r=p$ and $r=p+1$,  
$|\mathcal{T}_n^{(r,p)} | = C_{\floor{(n-1)/2},\floor{(n-1)/2}}$.  
\end{corollary}

We next enumerate permutations in $S_n$ which are $(r,p)$ toppleable for every $r$ with arbitrary but fixed $p$. This generalizes \cite[Theorem 2.4]{AHT} for $p = \ceil{n/2}$, where these were called toppleable permutations.
\begin{theorem}
\label{thm:top-perms-fixed-p}
Fix a positive integer $n$ and $p$, $1 \leq p \leq n$. Then 
$\pi \in \ds\bigcap_{r=1}^{n+1} \mathcal{T}_n^{(r,p)}$ if and only if $p + i - n \leq \pi^{-1}_i \leq p + i - 1$, $1 \leq i \leq n$.
Further, the number of such permutations is $C_{p,n-p}$.
\end{theorem}

\cref{thm: top-pB,thm:r-p-top-perm,cor:aht-theorem,thm:top-perms-fixed-p}  will be proved in \cref{sec:toppleable}.
We next study the toppling process from another point of view. We have seen that not every permutation topples to the identity permutation, but it is possible to characterize those who topple to it. On the other hand, it is also clear that not all permutations will be the final configuration after a toppling process. For instance, in case of 4 chips, by adding a chip on the second position, only the $4$ permutations $1234$, $1243$, $2134$, $2143$ occur out of the $24$; see \cref{table:data-p-resultant}. 

The question naturally arises as to which permutations arise as the result of a toppling process. Moreover, we would like to give a characterization of permutations that end in a certain permutation. 

\begin{definition}
	We say that a configuration $C \in \mathcal{S}(n,p)$ \emph{topples to} the permutation  $\pi \in S_{n+1}$ if toppling $C$ results in $\pi$. For $\pi \in S_{n+1}$, if there exists a $C \in \mathcal{S}(n,p)$ which topples to $\pi$, we say that $\pi$ is a \emph{$p$-resultant permutation}. 
\end{definition}

The list of $2$-resultant permutations in $S_4$ and the number of configurations in $\mathcal{S}(3,2)$ that topple to them are given in \cref{table:data-p-resultant}.
\begin{table}[h]
\begin{tabular}{|c|c|c|}
	\hline
	\text{Resultant permutation} & \text{Configurations} & \text{Number} \\
	\hline
	1234 & 1(23)4, 1(24)3, 1(34)2, 2(13)4, 2(14)3, 3(12)4, 3(14)2 & 7\\
	1243 & 4(12)3, 4(13)2 & 2 \\
	2134 & 2(34)1, 3(24)1 & 2 \\
	2143 & 4(32)1 & 1 \\
	\hline
\end{tabular}
\smallskip
\caption{The $2$-resultant permutations in $S_4$ on the left, configurations toppling to them in the middle, and the number of configurations on the right.}
\label{table:data-p-resultant}
\end{table}

The characterization of $p$-resultant permutations is then as follows.

\begin{theorem}
\label{thm:p-resultant-perms}
A permutation $\pi \in S_n$ is a $p$-resultant permutation if and only if $(\pi_1, \dots,\allowbreak \pi_{n-p})  \in S_{n-p}$.
\end{theorem}

A permutation $\pi$ in $S_n$ is \emph{irreducible}~\cite{klazar-2003} or \emph{indecomposable}~\cite{king-2006}
if there does not exist an $m$, $1 \leq m < n$ such that $\pi([m]) = [m]$. A permutation which is not irreducible is called \emph{reducible} or \emph{decomposable}.
\cref{thm:p-resultant-perms} says that $p$-resultant permutations are certain reducible or decomposable permutations.

To count the number of configurations toppling to a given $p$-resultant permutation, we recall that a \emph{left-to-right maximum} or a \emph{record} of a permutation $\pi$ is a value $\pi_j$ such that $\pi_j = \max\{\pi_1,\dots,\pi_j\}$. By convention, $\pi_1$ is taken to be a left-to-right maximum. It is a standard fact that the number of permutations in $S_n$ with $k$ left-to-right maxima are given by $\st{n}{k}$, 
the (unsigned) \emph{Stirling number of the first kind}; see \cite[Problem 6.63]{knuth-graham-patashnik-1994} for example. Similarly, one can define a \emph{right-to-left minimum} for a permutation. If $\pi_i = j$ is a left-to-right maximum for $\pi$, then one can show that $\pi^{-1}_j = i$ is a right-to-left minimum for $\pi^{-1}$. Therefore, the number of permutations of $S_n$ with $k$ right-to-left minima is also $\st nk$. 

For a given $p$-resultant permutation $\pi$, let $\pi^{\text{left}}$ be the induced permutation on $[n-p]$ and $\pi^{\text{right}}$ be the induced permutation on $\{n-p+1, \dots, n\}$.
The point of these definitions is the following lemma.

\begin{lemma}
\label{lem: topp_lrm}
The number of configurations toppling to a given $p$-resultant permutation $\pi \in S_n$ depends solely on the number of left-to-right maxima of $\pi^{\text{left}}$ and the right-to-left minima of $\pi^{\text{right}}$.
\end{lemma}

For general $n$ and $p$, we form an array $T^{(p)}_n$ with $n-p$ rows and $p$ columns. The $i$'th row of $T^{(p)}_n$ is indexed by permutations of $[n-p]$ with $i$ left-to-right maxima. The $j$'th column of $T^{(p)}_n$ is indexed by permutations of $\{n+1-p, \dots, n\}$ with $j$ right-to-left minima. Let the $(i,j)$'th entry of $T^{(p)}_n$ be the number of configurations that topple to any permutation $\pi$ with $i$ left-to-right maxima in $\pi^{\text{left}}$ and $j$ right-to-left minima in $\pi^{\text{right}}$.
\cref{lem: topp_lrm} guarantees that this is well-defined. For example, see \cref{table:3}.

\begin{table}[h]
	\begin{center}
	\begin{tabular}{|c||c|c|}
		\hline
		$\pi^{\text{left}} \backslash \pi^{\text{right}}$ & 65 & 56 \\
		\hline \hline
		4123, 4132, 4213, 4231, 4312, 4321 & 1 & 2 \\
		\hline
		1423, 1432, 2143, 2413, 2431 & $\multirow{2}{*}{2}$ & $\multirow{2}{*}{7}$ \\
		3124, 3142, 3214, 3241, 3412, 3421 & &\\
		\hline
		1243, 1324, 1342, 2134, 2314, 2341 & 4 & 23 \\
		\hline
		1234 & 8 & 73 \\
		\hline
	\end{tabular}
\end{center}
	\smallskip
	\caption{The array $T^{(2)}_6$ as defined after \cref{lem: topp_lrm}.}
	\label{table:3}
\end{table}

Our next theorem gives a statement about the size of the set of configurations that topple to a permutation in a class, i.e. with a given number of left-to-right maxima in $\pi^{\text{left}}$ and given number right-to-left minima of $\pi^{\text{right}}$. 

\begin{theorem}
\label{thm:p-resultant-configs}
The number of configurations that topple to a resultant permutation $\pi$ with $i$ left-to-right maxima in $\pi^{\text{left}}$ and $j$ right-to-left minima of $\pi^{\text{right}}$ is $\frac 12 B_{i,j}$.
\end{theorem}

To have the complete picture about resultant permutations, we now focus on the distinguished chip $r$ in an initial configuration. More precisely, we consider configurations of the form $\sigma^{r,p}$ as $\sigma$ varies. According to \cref{thm:p-resultant-perms} and \cref{lem: topp_lrm}, the resultant permutation $\pi$ splits up into $\pi^{\text{left}}$ of $[n-p]$ and $\pi^{\text{right}}$ of $\{n-p+1,\dots,n\}$. However, not all permutations of $[n-p]$ and $\{n-p+1,\dots,n\}$ will necessarily appear.
 
\begin{theorem}
	\label{thm:p-resultant-perms-r}
	A $p$-resultant permutation $\pi \in S_n$ is obtained by toppling a permutation by adding the chip $r$ if and only if $\pi$ satisfies the following conditions:
	\begin{enumerate}
			\item If $r \leq n-p$, then $r$ is a left-to-right maximum of $\pi^{\text{left}}$,
		
		\item If $r > n-p$, then $r$ is a right-to-left minimum of $\pi^{\text{right}}$.
		
	\end{enumerate}
\end{theorem}

\cref{thm:p-resultant-perms,lem: topp_lrm,thm:p-resultant-configs,thm:p-resultant-perms-r} will be proved in \cref{sec:finalconfigs}.
The last theorem determines the number of permutations $N_\pi \equiv N_\pi(r,p)$ toppling to a given permutation for a fixed value of $r$ and $p$. 
Before stating the next result for $p$-resultant permutations in $S_n$ obtained by adding the chip $r$ to site $p$, we give some data for $n=6$ and $p = r = 2$ in \cref{table:4}.
\begin{table}[h]
	\begin{center}
	\begin{tabular}{|c||c|c|}
		\hline
		$\pi^{\text{left}} \backslash \pi^{\text{right}}$ & 65 & 56 \\
		\hline \hline
		2143, 2413, 2431 & 2 & 4 \\
		\hline
		2134, 2314, 2341 & 4 & 14 \\
		\hline
		1243 & 2 & 10 \\
		\hline
		1234 & 4 & 32 \\
		\hline
	\end{tabular}
\end{center}
	\smallskip
	\caption{The $2$-resultant permutations of $S_4$ obtained by adding the chip $r=2$ along with the number of permutations toppling to them written as an array in the style of \cref{table:3}.}
	\label{table:4}
\end{table}
Using \cref{prop:symm}, it suffices to look at $r \leq n-p$.

\begin{theorem}
\label{thm:num-p-resultant-perms}
	Suppose $r \leq n - p$ and $\pi$ is a $p$-resultant permutation that satisfies the conditions of \cref{thm:p-resultant-perms-r}. 
	If the left-to-right maxima in $(\pi_1, \dots, \pi_{n-p})$ are $\{i_1,\dots,i_a,r,j_1,\allowbreak \dots,j_b\}$ in increasing order and there are $k$ right-to-left minima in $(\pi_{n-p+1}, \dots, \pi_{n})$, $N_\pi = \Delta^a \big( B_{b,k} \big)$, where $\Delta$ acts on the first index. 
\end{theorem}
	
\begin{proof}
This follows from \cref{thm:r-p-top-perm} and the bijection in the proof of \cref{thm:p-resultant-configs}.
\end{proof}

As a corollary of \cref{thm:num-p-resultant-perms}, we mention the special case where poly-Bernoulli numbers of type C occur. 
The permutations $\pi$ which occur and the corresponding $N_\pi$'s for $n=6$, $p = 3$ for both $r=3,4$ are listed in \cref{table:5}.
\begin{table}[h]
	\begin{center}
	\begin{tabular}{|c||c|c|c|}
		\hline
		$\pi^{\text{left}} \backslash \pi^{\text{right}}$ & 564, 654 & 465, 546, 645 & 456 \\
		\hline \hline
		312, 321 & 1 & 1 & 1 \\
		\hline
		132, 213, 231 & 1 & 3 & 7  \\
		\hline
		123 & 1 & 7 & 31\\
		\hline
	\end{tabular}
\end{center}
	\smallskip
	\caption{
		The $3$-resultant permutations of $S_6$ obtained by adding either the chip $r=3$ or $r=4$ along with the number of permutations toppling to them written as an array in the style of \cref{table:3}.}
	\label{table:5}
\end{table}

\begin{corollary}
	Suppose $r = n - p$ and $\pi$ is a $p$-resultant permutation that satisfies the conditions of \cref{thm:p-resultant-perms-r}. Assume that there are $i$ left-to-right maxima 
	in $(\pi_1, \dots, \pi_{n-p})$ and $j$ right-to-left minima in 
	$(\pi_{n-p+1}, \dots, \pi_{n})$. Then 
	$N_\pi = C_{j-1,i-1}$.
\end{corollary}

\begin{proof}
This follows from \cref{cor:aht-theorem} and the bijection in the proof of \cref{thm:p-resultant-configs}.
\end{proof}

Using \cref{prop:symm}, we observe that the poly-Bernoulli numbers of type C occur when $r = n-p$ or $n-p+1$.

\section{Poly-Bernoulli numbers}
\label{sec:poly-bernoulli}

Here, we introduce the combinatorial sequences that will play the main role in this work.
Poly-Bernoulli numbers were introduced by M. Kaneko~\cite{KanekoIntro} as a generalization of the classical Bernoulli numbers during his investigations of multiple zeta values. 

\subsection{Poly-Bernoulli numbers of type B}

For integers $n>0$ and  $k$, poly-Bernoulli numbers of type B were originally  defined by the generating function \eqref{bnk-gf}.
From the combinatorial point of view poly-Bernoulli numbers with negative indices are of interest, since these numbers are nonnegative integers and there are several different combinatorial objects that are enumerated by these numbers. The array appears in OEIS as \cite[A099594]{OEIS}. 

The combinatorics of poly-Bernoulli numbers is very rich. The first combinatorial interpretation was given by Brewbaker \cite{Brewbaker} as the number of $n\times k$ \emph{lonesum matrices}, which are $\{0,1\}$-matrices uniquely determined by their row and column sum vectors. We now give some combinatorial objects counted by poly-Bernoulli numbers of type B that are relevant to the toppling process. For further examples, see \cite{BH} for instance.

Permutations with the restriction on the difference between position and value are well studied \cite{benyi-hajnal-2017,KimKrotovLee,
Launois,Sjostrand,vesztergombi-1974,lovasz-vesztergombi-1978}. 
We follow the convention of \cite{BH} and refer to them as \emph{Vesztergombi} permutations since they were first studied by Vesztergombi~\cite{vesztergombi-1974}.

\begin{definition}
\label{def:veszt}
Let $1 \leq k,n$. A permutation $\pi \in S_{k + n}$ is said to be a \emph{$(k,n)$-Vesztergombi permutation} if $-k \leq \pi_i - i \leq n$ for $1 \leq i \leq k + n$.
\end{definition}

Such permutations were first studied by Vesztergombi~\cite{vesztergombi-1974}.
An example of a $(9,6)$-Vesztergombi permutation in two-line notation is  
\begin{align}
\label{eg-vesz}
	\sigma =\left(
	\begin{array}{ccccccccccccccc}
		1 & 2& 3& 4& 5& 6&7&8&9&10&11&12&13&14&15\\
		1&6&4&8&7&10&12&11&13&3&2&9&5&14&15
	\end{array}\right).
\end{align}

\begin{theorem}[{\cite{Launois}}]
\label{thm:num-veszt}
The number of $(k,n)$-Vesztergombi permutations is $B_{n,k}$.
\end{theorem}

The double exponential generating function of poly-Bernoulli numbers is given by the elegant expression~\cite{KanekoIntro}
\begin{align*}
		\sum_{k=0}^{\infty}\sum_{n=0}^{\infty}B_{n,k}\frac{x^n}{n!}\frac{y^k}{k!} =
		\frac{e^{x+y}}{e^x + e^y -e^{x+y}}.
\end{align*}
Recall that $\sts{n}{m}$ is the \emph{Stirling number of the second kind}, which enumerates, among other things, set partitions of $[n]$ with $m$ parts.
The following three basic formulas were proven combinatorially in the literature:
\begin{itemize}
	\item[1.]
	the closed formula~\cite{BH,Brewbaker},
	\begin{align}
	\label{bnk-sum}
		B_{n,k} =  \sum_{m=0}^{\min(n,k)} (m!)^2 \sts{n+1}{m+1} \sts{k+1}{m+1},
	\end{align}
	
	\item[2.] the inclusion-exclusion type formula~\cite{Brewbaker},
	\begin{align}
	\label{bnk-ie}
			B_{n,k}=\sum_{m=0}^n (-1)^{n-m} m! \sts{n}{m}(m+1)^k, 
	\end{align}
	
	\item[3.] and the recurrence relation~\cite{BH},
	\begin{align}
	\label{bnk-recur}
			B_{n,k+1} = B_{n,k} + \sum_{m=1}^n\binom{n}{m}B_{n-(m-1),k}.
	\end{align}
\end{itemize}
The asymptotic for the diagonal entries is given in~\cite{Lundberg_Khera}
\begin{align*}
	B_{n,n} \sim \sqrt{\frac{1}{n\pi(1-\log 2)}}
	\frac{(n!)^2}{\left(\log 2\right)^{2n+1}}.
\end{align*}

We now define two other classes of objects counted by the poly-Bernoulli numbers. The first definition appears in \cite[Sequence~A099594]{OEIS} by D. Callan, hence the nomenclature in \cite{BH}.

\begin{definition}
\label{def:callan}
Let $1 \leq k,n$. A permutation $\pi \in S_{k + n}$ is said to be a \emph{$(n,k)$-Callan permutation} if each maximal contiguous substring whose support belongs to $\{1, 2,\ldots, n\}$ (resp. $\{n+1, n+2, \ldots, n+k\}$) is increasing (resp. decreasing). 
\end{definition}

In order to emphasize the two types in the set we sometimes distinguish the elements $\{1, 2,\ldots, n\}$ with an underline and the elements $\{n+1, n+2, \ldots, n+k\}$ with an overline. In this terminology, a Callan permutation is an alternating sequence of increasing underlined and decreasing overlined subsequences.

\begin{example} 
The set of $(2,2)$-Callan permutations ordered according to the first letter is listed below:
	\begin{align*}
		&\piros{12}\kek{43};\quad \piros{1}\kek{43}\piros{2};\quad \piros{1}\kek{3}\piros{2}\kek{4};\quad \piros{1}\kek{4}\piros{2}\kek{3}; \\
		&\piros{2}\kek{43}\piros{1};\quad   \piros{2}\kek{3}\piros{1}\kek{4};\quad  \piros{2}\kek{4}\piros{1}\kek{3};\\
		&\kek{3}\piros{12}\kek{4}; \quad \kek{3}\piros{1}\kek{4}\piros{2}; \quad  \kek{3}\piros{2}\kek{4}\piros{1};\\
		&\kek{4}\piros{12}\kek{3};\quad \kek{4}\piros{1}\kek{3}\piros{2}; \quad \kek{4}\piros{2}\kek{3}\piros{1};\quad \kek{43}\piros{12}.
	\end{align*}
\end{example}

\begin{definition}
\label{def:ao}
For any simple, undirected graph, an {\em orientation} is an assignment of arrows to the edges. An {\em acyclic orientation} (AO) is an orientation in which there is no directed cycle. It is easy to see that every graph has an acyclic orientation and every acyclic orientation has at least one {\em source} (vertex with no incoming arrows) and one {\em sink}
(vertex with no outgoing arrows).
\end{definition}

Recall that the complete bipartite graph $K_{m,n}$ is a graph with $m+n$ vertices such that the first $m$ vertices are adjacent to the last $n$ vertices, and there are no other edges.
The following result is a combination of several results in the literature.

\begin{proposition}
\label{prop:bij-callan}
The poly-Bernoulli number $B_{n,k}$ counts the following objects:
\begin{enumerate}
\item AOs of $K_{n,k}$,
\item $(n,k)$-Callan permutations.
\end{enumerate}
\end{proposition}

\begin{proof}
Part (1) is due to \cite[Theorem 2.1]{cameron2014acyclic}.
We prove (2) by constructing an explicit bijection between 
$(k,n)$-Callan permutations $\pi$ and $(k,n)$-Vesztergombi permutations $\sigma$ such that $\pi_1 = \sigma^{-1}_{n+1}$.

We sketch here a bijection 
along with an example. Since this bijection is essentially the translation of the bijection between Vesztergombi permutations and lonesum matrices given by Kim--Krotov--Lee~\cite[Appendix]{KimKrotovLee}, we will be sketchy. 
We refer the reader to the original paper for the details.  

Let $\pi$ be a $(9,6)$-Callan permutation, i.e., the number of underlined elements is $k=9$ and the number of overlined elements is $n=6$.
\begin{align*}
	\pi =\piros{5},\piros{7},\kek{12},\kek{11},\piros{1},\piros{4}, \piros{8}, \kek{14},\piros{3},\piros{6}, \piros{9}, \kek{15},  \kek{13},\kek{10}, \piros{2}.
\end{align*}

We determine the Vesztergombi permutation $\sigma \in S_{n+k}$ 
in the two-line notation for a permutation.
First, we describe how to associate pairs $(i, \sigma_i)$ to underlined elements. If the starting block is underlined, say $A_0=\{a_{0,1},a_{0,2},\ldots,a_{0,{\ell}_0}\}$, then we define the pairs: $(a_{0,1},n+1)$, $(a_{0,2},a_{0,1}+n+1)$, $\ldots$, $(a_{0,\ell_0-1},a_{0,\ell_0}+n+1)$. 
In our running example, we have a starting underlined block $\piros{5},\piros{7}$, so we get the pairs $(5,7)$ and $(7,12)$.
For any other underlined block, $A_i=\{a_{i,1},a_{i,2},\ldots,a_{i,{\ell}_i}\}$, we define $(a_{i,2}, a_{i,1}+n+1)$, $(a_{i,3},a_{i,2}+n+1)$, $\ldots$, $(a_{i,\ell_{i}},a_{i,\ell_{i-1}}+n+1)$. In our case, 
$\piros{1}, \piros{4}, \piros{8}$ determines $(4,8)$ and $(8,11)$,  $\piros{3},\piros{6}, \piros{9}$ give $(6,10)$ and $(9,13)$. 
The leading elements $r_{i,1}$ (in our case $\piros{1}, \piros{2}$ and $\piros{3}$) will be dealt later. 
A similar rule is applied for the overlined blocks.
Given an overlined block $B_i=\{b_{i,1},b_{i,2},\ldots,b_{i,{\ell}_i}\}$, we define $(b_{i,2}, b_{i,1}-k-1)$, $(b_{i,3},b_{i,2}-k-1)$, $\ldots$, $(b_{i,\ell_{i}},b_{i,\ell_{i-1}}-k-1)$. 
In our case,
\begin{align*}
	\kek{12},\kek{11}&\rightarrow (11, 2),\\
	\kek{15},  \kek{13},\kek{10}&\rightarrow (13, 5),\,(10, 3).
\end{align*}
The leadings elements here are $\kek{12}$, $\kek{14}$, $\kek{15}$. 

In our running example, we have so far,
\begin{align*}
\sigma =
	\left(
	\begin{array}{ccccccccccccccc}
		1 & 2& 3& 4& 5& 6&7&8&9&10&11&12&13&14&15\\
		&&&8&7&10&12&11&13&3&2&&5&
	\end{array}\right)
\end{align*}
We now have to fill in the remaining elements.
It can be checked that there are exactly $m$ such "missing" numbers from $\{1,\ldots, k\}$ and also $m$ from $\{k+1,\ldots, n+k\}$. 
Let $c_1<c_2<\cdots< c_m$ be the "missing" numbers from $\{1,\ldots, n\}$ and $d_1<\ldots< d_m$ the missing numbers from $\{n+1,\ldots, n+k\}$. 
In the example $\{c_1,c_2,c_3\}=\{1,4,6\}$ and $\{d_1,d_2,d_3\}=\{9,14,15\}$. 

We connect the leading elements of the blocks and the missing elements as follows: $(a_{i,1},c_i)$ and $(b_{i,1},d_i)$. 
Doing so in our running example we obtain
\begin{align*}
	(1,1)\, (3,4)\, (2,6)\, (12,9)\, (14,14)\, (15,15)
\end{align*}
and so these pairs together give the Vesztergombi permutation $\sigma$ in \eqref{eg-vesz}.  
\end{proof}

\begin{remark}
\label{rem:bnk-even}
We can see that $B_{n,k}$ are even integers for $n,k > 0$ as follows. It is trivial in any combinatorial model that $B_{n,1}$ are powers of $2$. For instance, in terms of Callan permutations, $B_{n,1}$ counts permutations where all consecutive entries are in increasing order except the element $n+1$. So, if $i$ denotes the number of elements to the left of $n$, we can count by $\sum_{i=0}^{n}\binom{n}{i} = 2^n$ the different permutations with such a property. 
The parity of $B_{n,k}$ for $k>1$ then follows from the recurrence \eqref{bnk-recur}.
\end{remark}

\subsection{Poly-Bernoulli numbers of type C}

Poly-Bernoulli numbers of type C were introduced by Kaneko analytically for general $k$ by the generating function in \eqref{cnk-gf}. 
Several combinatorial sequences are enumerated by these numbers as well when $k$ is a nonnegative integer. We list them below.

\begin{definition}
\label{def:exc}
An \emph{excedance} of a permutation $\pi$ is a position $i$ such
that $\pi_i > i$. The positions at which there are excedances for $\pi$ is
called the \emph{excedance set} of $\pi$. 
\end{definition}

Recall the definition of Vesztergombi permutations from \cref{def:veszt}, Callan permutations from \cref{def:callan} and acyclic orientations from \cref{def:ao}.

\begin{theorem}[{\cite[Theorems 12, 10, 19, 16]{benyi-hajnal-2017}}]
\label{thm:polyC-objects}
The poly-Bernoulli number $C_{n,k}$ counts the following objects:
\begin{enumerate}
\item the number of permutations in $S_{n+k}$ having excedance set $[k]$,

\item the number of permutations $\pi$ in $S_{n+k}$ with $-k\leq \pi(i)-i<n$, 

\item the number of AOs of $K_{n,k}$ with a unique sink,

\item $(n,k)$-Callan permutations that start with an underlined element.

\end{enumerate}
\end{theorem}

The bivariate exponential generating function for $C_{n,k}$ is given by
\begin{align*}
	\sum_{n=1}^{\infty}\sum_{k=1}^{\infty}C_{n,k}\frac{x^n}{n!}\frac{y^k}{k!}=
	\frac{e^{x}}{e^x+e^y-e^{x+y}}.
\end{align*}
The analogue formulas to the ones of the poly-Bernoulli numbers of type B are as follows \cite{benyi-hajnal-2017}: 

\begin{itemize}
	\item[1.]
	the closed formula,
	\begin{align}
	\label{cnk-sum}
		C_{n,k}=\sum_{m=0}^{\min(n,k)}(m!)^2\sts{n+1}{m+1}\sts{k}{m},
	\end{align}
	
	\item[2.] the inclusion-exclusion type formula,
	\begin{align}
	\label{cnk-ie}
		C_{n,k}= \sum_{m=0}^n(-1)^{n+m}m!(m+1)^k\sts{n+1}{m+1}, 
	\end{align}
	
	\item[3.] and the recurrence relation
	\begin{align}
	\label{cnk-recur}
	C_{n,k+1}=\sum_{m=1}^{n}\binom{n}{m}C_{n-m+1,k},\qquad n\geq 1,\quad k\geq 0.
	\end{align}
\end{itemize}

The relation between the two types of poly-Bernoulli numbers can be expressed by the following equations.

\begin{align}\label{relationB_C}
	B_{n,k} &= \sum_{i=0}^k\binom{k}{i} C_{n,i}  \quad \mbox{and} \quad C_{n,k} = (-1)^n\sum_{i=0}^{n} (-1)^i\binom{n}{i}B_{i,k},
\end{align}
\begin{align*}
	B_{n,k} =C_{n,k}+C_{n+1,k-1}.
\end{align*}
The asymptotic for the diagonal entries is given in \cite{Lundberg}

\begin{align*}
	C_{n,n} \sim \left(\frac{1}{2\log 2 \sqrt{1-\log 2}}+o(1)\right)
	\frac{n!}{\left(2\log 2\right)^n} .
\end{align*}

\section{Toppleable configurations and permutations}
\label{sec:toppleable}

In this section, we will be interested in configurations and permutations that get sorted after the toppling process, i.e., the final configuration is the identity permutation. The data for the number of toppleable configurations is given in \cref{tab:topp-confs}.

\begin{table}[h!]
\begin{center}
\begin{tabular}{c|cccccc}
$n \backslash p$ & 1 & 2 & 3 & 4 & 5 & 6 \\
\hline
1 & 1 \\
2 & 2 & 2 \\
3 & 4 & 7 & 4 \\
4 & 16 & 73 & 115 & 73 & 16 & \\
5 & 32 & 227 & 533 & 533 & 227 & 32
\end{tabular}
\caption{The number of toppleable configurations in $\mathcal{S}(n,p)$ for small values of $n$ and $p$.}
\label{tab:topp-confs}
\end{center}

\end{table}

Since the order of the topplings does not influence the final configuration by \cref{prop:determ}, we can define a special order of the topplings, that is easy to analyze. This idea leads to the notion of a \emph{pass} that we recall from \cite{AHT}. 
For the sake of simplicity, we focus now only on the number of chips at each site. We start with the unlabeled configuration $(\_,1,\ldots, 1,1,\hat{2},1,1,\ldots, 1,\_)$ on $L_n$ where the hat denotes the site $p$. The first toppling is necessarily on site $p$, leading to $(\_,1,\ldots, 1,2,\hat{\_},2,1,\ldots, 1,\_)$. Next, we topple the chips on the sites to the left of $p$, (the $(p-1)$'th site) and to the right of $p$ (the $(p+1)$'th site) to get $(\_,1,\ldots, 1,2,\_,\hat{2},\_,2,1,\ldots, 1,\_)$. From now on we leave the two chips on site $p$, while we topple the chips on right and left, as long as it is possible. 
The sequence of these topplings is called the \emph{first pass}.
After the first pass we end up with $(1,\_,1,\ldots, 1,1,\hat{2},1,1,\ldots, 1,\_, 1)$. 
Clearly, if $p$ is not in the center, on one side we will have more topplings. However, the first pass includes all together $n+1$ topplings. 

Similarly, we perform the \emph{second pass} starting with the toppling of the two chips at site $p$, and continue similarly as before ending up with $(1,1,\_,1,\ldots, 1,1,\hat{2},1,1,\ldots, 1,\_, 1,1)$. We continue this way, settling the topplings in passes. 
After $\min(p,n+1-p)$ passes, we will arrive at a final configuration in which no site contains more than one chip.

The intermediate configuration after every pass can be decomposed into three parts depending on the locations of the empty sites. We call the part of the configuration to the left of the first empty site $(1,1,\ldots, 1,\_)$ the \emph{left arm}, the part to the right of the second empty site $(\_,1,1,\ldots, 1)$ the \emph{right arm}, and the part between the two empty parts the \emph{active part}. 

We now list some important observations about these pass moves which follow from \cite{AHT} when considering the toppling process on a configuration in $\mathcal{S}(n,p)$.

\begin{remark}
\label{rem:passes}

\begin{enumerate}
\item At the end of any pass, the chips in the left and right arms are `frozen', i.e. these chips do not change their positions in any topplings or passes thereafter. 

\item Every chip in the active part topples at least once during a pass.

\item 

After the last pass, there are only two parts: the left arm containing a chip each on the sites $0, 1,\ldots, n-p$, and the right arm containing a chip each on the sites $n-p+2,\ldots, n+1$.

\item Assume $a$ and $b$, with $a<b$, are the chips at the site $p$ at the beginning of a pass. In a pass, $a$ propagates to the left, until it meets a smaller element, $a_1$ say. At that point, $a$ is stuck at that position, and $a_1$ moves to the left until it meets a smaller element, and so on. Similarly $b$ moves to the right, until it meets a greater element, $b_1$ say, after which it gets stuck. Then $b_1$ moves right until it meets a greater element and so on. 

To summarize, to the left of site $p$, chips $a>a_1>\cdots>a_j$ move some positions to the left, while all the other chips are just shifted by one position to the right without changing their relative positions. Similarly, to the right of $p$, chips $b<b_1<\cdots b_k$ move some positions to the right, while all the other chips are only shifted by one position to the left without changing their relative positions.   
\end{enumerate}
\end{remark}

We define some more terminology in order to be able to talk about the toppling process more precisely. 
We refer to the active sites $i$ with $i \leq p$ as the \emph{left part}
and those active sites $i$ with $i \geq p$ as the \emph{right part}. See \cref{fig:ithpass} for an illustration of the terminology used in the proofs.

\begin{center}
\begin{figure}[h!]
\includegraphics[scale=1]{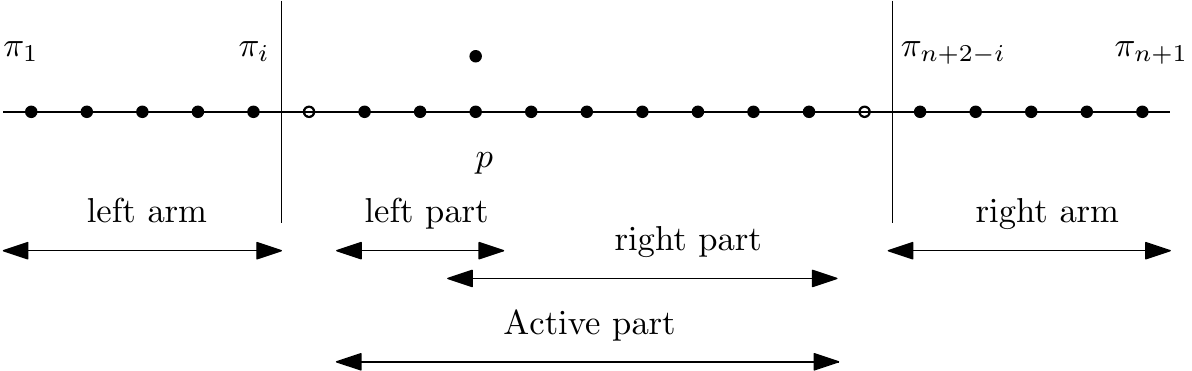}
\caption{Summary of the terminology used in the discussion regarding what happens at the end of the $i$'th pass for a configuration in $\mathcal{S}(n,p)$.}
\label{fig:ithpass}
\end{figure}
\end{center}
 
Let $S_0$ denote the set of chips that are on sites $i\leq p$, i.e. set of the chips in the left part and the two chips at site $p$ in the initial configuration.  Similarly, let $T_0$ be the set of chips in the right part and on site $p$ in the initial configuration.
		In general, let $S_j$ be the set of chips on the sites $\leq p$, and $T_j$ be the set of chips on the sites $\geq p$ after the $j$'th pass. It is easy to see that the least element in $S_0$ will topple during the first pass to the first position as the first element in the left arm, becoming the first element of the resulting sorting, $\pi_1$. Similarly, the greatest element of $T_0$ will topple to the last site during the first pass and be frozen there.

By \cref{rem:passes}(1), if a configuration is toppleable then chips $i$ and $n-i+1$ have to be in their correct positions after the $i$'th pass. This fact restricts the positions of the chips in the initial configuration in a toppleable configuration. As we will see in the next proposition this property characterizes toppleable configurations. The idea of the proof is similar to that of the classification in \cite[Theorem 3.4]{AHT} up to a point.
The novelty here is that the number of passes is potentially much smaller than $\floor{n/2}+1$.

For a configuration $C$, let $C^{-1}(i)$ denote the position of the chip labeled by $i$ in the initial configuration.  

\begin{theorem}
\label{thm:config_char}
 	A configuration $C \in \mathcal{S}(n,p)$ is toppleable if and only if for $1\leq i\leq n+1$
 	\begin{align}\label{condition}
 		p+i-n-1\leq C^{-1}(i) \leq p+i-1.
 	\end{align}
\end{theorem}

It will be useful to state two lemmas explicitly as a preparation of the proof of the theorem.
The proof of the first one, \cref{lem:j_pass_char} 
is directly based on the proof in \cite{AHT}. 
Let $C^{-1}_j(i)$ denote the position of the chip $i$ after the $j$'th pass.

\begin{lemma}
\label{lem:j_pass_char}
Suppose that we have a configuration $C \in \mathcal{S}(n,p)$.
If $p+i-n-1\leq C^{-1}(i) \leq p+i-1$ holds for all chips $i$ then after $j$ ($j\leq p$) passes, the chips $1,2,\ldots, j$  and  $n+1, n,\ldots, n-j+2$  are fixed in their correct positions. 
Furthermore, for the chips  $i=j+1, j+2, \ldots, p$, 
$C^{-1}_j(i)<p+i-j$, 
while for $i=n-j+1,n-j,\ldots p$.
$C^{-1}_j(i)>p+i-1-n+j$.
\end{lemma}

\begin{proof}
	The lemma states that the distance between the chip labeled by $i$ and the $i$'th site is reduced by at least one in each pass.
	
	We argue using induction on $j$. The base case is $j=0$ which is fulfilled by assumption. Now suppose the statement holds for the pass $j-1$, and consider the $j$'th pass. For chip $i=j$, we have $C^{-1}_{j-1}(j) < p+j-(j-1) = p+1$, which means $C^{-1}_{j-1}(j)\leq p$. 
	Recall from \cref{rem:passes}(3) that a chip to the left of $p$ moves to the left arbitrary many (more precisely, the number of consecutive greater elements to its left)  sites to the left and at most one site to the right. 
	
	Since $j$ is the smallest non-fixed chip from $C^{-1}_{j-1}(j)\leq p$, $j=\min S_{j-1}$. Hence we deduce that in the $j$'th pass, it will topple to the left until it is fixed, and so $\pi_j=j$. Thus, after $j$ passes, the chips $1,2,\ldots, j$ are fixed in the right order. 
	
	For chips $i=j,\ldots, p$, we use an inductive argument. 
We show that during the $j$'th pass the chip $i$ does not move to the right of site $p+i-j$ and must land strictly to the left of the site $p+i-j$. The base case is $i=j$ and it follows from the previous consideration.
	Assume now that the statement holds for $j, j+1, \ldots, i-1$.  Note that from the outer induction we have $C^{-1}_{j-1}(i)<p+i-j+1$. We want to show that  after the $j$'th pass the chip $i$ is to the left of the site $p+i-j$.
	\begin{itemize}
		\item if $C^{-1}_{j-1}(i)<p$, then the chip $i$ can move in the $j$'th pass at most one site to the right. Then it is still strictly to the left of $p+1$, which is included in $p+i-j$ for the values of $i$ considered here. 
		\item if $C^{-1}_{j-1}(i)\geq p$ then the chip $i$ can move to the right multiple times. There are two possibilities here: either it never reaches the site $p+i-j$, in which case we are done, or it reaches (or starts) at the site $p+i-j$. Note that according to the induction hypothesis no chips smaller than $i$ can reach the site $p+i-j$. Thus, chip $i$ "meets" a greater chip at this site, and topples to the left during the $j$'th pass, landing to the left of the site $p+i-j$.
	\end{itemize}
	Hence, in either case, we have shown that $C^{-1}_{j}(i)<p+i-j$.
	
	A similar argument shows that the chips $n+1, n,\ldots, n-j+2$  (where $j\leq p$) are fixed in their correct positions after the $j$'th pass and for the chips $i=n-j+1, n-j, \ldots, p$ the condition $C^{-1}_j(i)>p+i-1-n+j$ holds. 
\end{proof}

Our next lemma shows that we can say even more about the positions of the chips after a given number of passes based on their initial positions. 

\begin{lemma}
\label{lem:j_pass_char_2}
Suppose that we have a configuration $C \in \mathcal{S}(n,p)$.
	If $p+k-n-1\leq C^{-1}(k) \leq p+k-1$ holds for all chips $k$ and $i\leq p-1$, we  have $p+i-1-n+j<C^{-1}_j(i)$.
\end{lemma}

\begin{proof}

	To prove this inequality we use now induction on the chips $i=p-1,p-2,\ldots$ and on the passes, $j$. 
	The initial case is when $i=p-1$ and $j=1$.
	We have to show that $p+(p-1)-1-n+1<C^{-1}_1(p-1)$.
	
	We focus on the moves of the chip $p-1$ during the first pass. By \cref{rem:passes}(2), chip $p-1$ topples at least once during the pass. If this toppling is to the right, we are done.
	But if either
	\begin{enumerate}
		\item[(a)] it is initially at the $p$'th site sharing this site with a greater chip,
		\item[(b)] or a greater chip topples during the pass to the site of the chip $p-1$,
	\end{enumerate}
	the chip $p-1$ will topple first to the \textit{left}.
	
	Moreover, if there is a greater chip to its left, then again topples to the left and so on, until it "meets" a smaller chip. In the last toppling, $p-1$ has to topple to the right and does not move in the pass anymore.
	Let $a_1$ denote the chip that is greater than $p-1$ and shares the site of $p-1$ when it topples first to the left. In case (a), this is the other chip on the site $p$ in the initial configuration, in case (b) this is a chip that topples to the site of $p-1$ from the right side.
	Let $a_2$ denote the next greater chip that $p-1$ meets and forces $p-1$ topple to the left again, and so on. 
	Let $a_k$ the last greater chip in this sequence, so that after the chips $a_k$ and $p-1$ shared a site, and since $p-1<a_k$, $p-1$ topples to the left, but the next chip $p-1$ meets is a smaller element. In other words on the site next to the left of the chip $a_k$ -- $C^{-1}(a_k)-1$ -- there is a chip $b<p-1$. 
	In the toppling where $b$ and $p-1$ is on the site $C^{-1}(a_k)-1$, the chip $p-1$ topples to the right, i.e. back to the original site of $a_k$, $C^{-1}(a_k)$.
	We assumed that the condition $p+i-n-1\leq C^{-1}(i) <p+i-1$ holds for all $i$, and so also for the chip $i=a_k$. 
	Therefore, we have
	\begin{align*}
		p+a_k-n-1\leq C^{-1}(a_k)\leq p+a_k-1.
	\end{align*}
	Since $a_k>p-1$, we have
	\begin{align*}
		p+(p-1)-n-1+1\leq p+a_k-n-1\leq C^{-1}(a_k)
	\end{align*}
	
	As we have seen that the chip $p-1$ will land at the position of $a_k$ after the first pass, it follows that after the first pass the position of $p-1$ is at least $p+p-1-n-1+1$, which is what we wanted to show.
	
	The same argument can be used for all the chips $i\leq p$ in the first pass, hence we have $p+i-1-n+1<C^{-1}_1(i)$ for all $i \leq p$. Exactly the same kind of argument can be used in all passes, which implies the general statement $p+i-1-n+j<C^{-1}_j(i)$. 
\end{proof}

Now we are ready to prove \cref{thm:config_char}.

\begin{proof}[Proof of \cref{thm:config_char}]

First, we show that if a configuration $C$ is toppleable, then \eqref{condition} holds. Let $C\in \mathcal{S}(n,p)$.  
Suppose  that for an $i$ with $1\leq  i\leq p$ the condition \eqref{condition} is not fulfilled, $C^{-1}(i)> p+i-1$. (The left hand side, $p+i-n-1\leq C^{-1}(i)$, holds again trivially for $i$ with $1\leq  i\leq p$).
As noted in \cref{rem:passes}(3), a chip to the right of $p$ (in a set $T_j$) moves during a pass arbitrary sites to the right, but at most one to the left. Hence, if $C^{-1}(i)> p+i-1$, the chip $i$ is after $i-1$ passes still strictly to the right of the site $p$. But, since in each pass one chip is fixed on the rightmost and leftmost active site, $\pi_i$ is fixed after the $(i-1)$'th pass. Hence, the chip that is fixed after the $(i-1)$'th pass cannot be the chip $i$. The case for $p<  i\leq  n+1$ follows by symmetry.  

We show now that if the condition \eqref{condition} holds for all $i$, the configuration is toppleable. 
Note that the toppling process stops after $\min(p,n-p+1)$ passes. Assume now that $n-p<p$; the other case follows by symmtery.
By \cref{lem:j_pass_char} we know that after the $(n-p+1)$'th pass, the chips $1,2,\ldots, n-p+1$, and $p,p+1,\ldots, n+1$ will be in their correct order. However, the question remains as to what happens with the chips $n-p+2,\ldots, p-1$. We show next that these chips topple also in the correct order after the $(n-p+1)$'th pass.

Let $i=n-p+t$, where $2\leq t\leq 2p-n-1$. The condition of the theorem has the form
\begin{align*}
    	t-1\leq C^{-1}(n-p+t)\leq n+t-1.
\end{align*}
     Since the right hand side is redundant, we have actually $t-1\leq C^{-1}(n-p+t)\leq n+1$. After $n-p$ passes, the restriction modifies according  by \cref{lem:j_pass_char} and \cref{lem:j_pass_char_2} for all $t$ to
     \begin{align*}
     	n-p+t-1\leq C^{-1}_{n-p}(n-p+t)\leq p.
     \end{align*}  
 
As in the proofs of the previous lemmas, we again use an inductive argument. 
Consider the base case $t=2p-n-1$, i.e., before the last pass the chip $p-1$ belongs to one of the sites in $\{p-2,p-1,p\}$. In the last pass it topples at least once by \cref{rem:passes}(2). If it is on site $p-2$ it has to topple to the right, since there are no greater elements that could force it to topple to the left (recall that the elements $p,\ldots, n+1$ are fixed or topple now to their correct places). 

If it is on the $(p-1)$'th or $p$'th site, one can see that it has to topple to its correct position after a little bit of thought. Essentially, after the $(p-1)$'th site there is simply
no other place where it could land, because the $p$'th and $(p+1)$'th sites will be necessarily occupied by the chips $p$ and $p+1$.

Finally, induction on $t$ completes the proof: assume that for $t,t+1,\ldots, 2p-n-1 $ the chips $n-p+t$ will land at the correct site after the $(n-p+1)$'th pass, so the sites $n-p+t,\ldots, p, p+1,\ldots, n+1$ will be occupied with the correct chips. Then before the last pass the chip $n-p+t-1$ was restricted between the sites $n-p+t-2$ and $p$. So there is no greater chip that could cause it to topple to the left from the $(n-p+t-2)$'th site. In addition, there is no ``free'' place where it could land beside its correct site. Thus, all the chips between $n-p+1$ and $p$ have to topple to their correct sites finally in the last pass, and this proves the result. 
\end{proof}

As explained in \cref{sec:model}, we can associate to each $p$-toppleable configuration $C=(c_1,\ldots, (c_{p}^1,c_{p}^2),\ldots, c_{n}) \in \mathcal{S}(n,p)$ two toppleable permutations; a $(c_p^1,p)$-toppleable and a $(c_p^2,p)$-toppleable one. 
Let us distinguish one chip labeled by $r$ on site $p$ in the configuration $C$ and call such configurations \emph{marked configurations}. Clearly, there is a one-to-one correspondence between the set of toppleable marked configurations on $[n+1]$ chips and $(r,p)$-toppleable permutations of $[n]$. 
To make this precise, we define a map $w$ that associates to each marked configuration $C_p^r=(c_1,\ldots, (c_p,{\bf r}), \ldots, c_n)$ on $n+1$ chips a permutation $w(C_p^r) = \sigma^*$ of  $[n+1]$ as follows: 
\[
\sigma^*_i= 
\begin{cases}
c_i & 1 \leq i \leq p, \\
r & i = p+1, \\
c_{i-1} & p+2 \leq i \leq n.
\end{cases}
\]
Using \cref{thm:config_char} and the map $w$ we can characterize $(r,p)$-toppleable permutations. Recall the definition of Vesztergombi permutations from  \cref{def:veszt} and the definition of Callan permutations from \cref{def:callan}.

\begin{proposition}
\label{prop: perm_char}
The set of $(r,p)$-toppleable permutations of $[n]$ is in one-to-one correspondence with $(n-p+1,p)$-Vesztergombi permutations $\sigma$
such that $\sigma^{-1}_{p+1} = r$ and with $(n-p+1,p)$-Callan permutations starting with the element $r$.
\end{proposition}

\begin{proof}
As we noted above, to each $(r,p)$-toppleable permutation one can associate a configuration $C_p^r \in \mathcal{S}(n,p)$ that topples to the identity. 
Using the map $w$, we obtain a permutation $\sigma^*$ of $[n+1]$ with $\sigma^*_{p+1}=r$.
The restriction on the positions of chips in \cref{thm:config_char} translates to a similar restriction on positions in $\sigma^*$ by shifting the right side inequality by 1, and we obtain 
$p-n-1\leq(\sigma^*)^{-1}_i-i\leq p$. 
Note that this is the same restriction as in the definition of $(p,n-p+1)$-Vesztergombi permutation. 
Thus $\sigma^*$ is a $(p,n-p+1)$-Vesztergombi permutation with $\sigma^*_{p+1}=r$. Thus, the inverse of $\sigma^*$, that we denote by $\sigma$, is a $(n-p+1,p)$-Vesztergombi permutation\footnote{One can easily see that 
the inverse of a $(k,m)$-Vesztergombi permutation is an $(m,k)$-Vesztergombi permutation.}  
with $\sigma^{-1}_{p+1}=r$, proving the first part.  
	
The second part follows from the bijection in \cref{prop:bij-callan}.
\end{proof}

We are now in a position to prove the first main result of this section.

\begin{proof}[Proof of \cref{thm: top-pB}]
The theorem follows from the characterization in \cref{prop: perm_char}, the bijection in \cref{prop:bij-callan} and from the result in \cref{thm:num-veszt} that Vesztergombi permutations are enumerated by the poly-Bernoulli numbers. 
\end{proof}

We now move towards proving \cref{thm:r-p-top-perm} using combinatorial arguments.
We first list the number of $(r,p)$-toppleable permutations for all values of $r$ in \cref{tab:n4 and 5rp}.

\begin{center}
\begin{table}[h!]
	\begin{tabular}{|c||c|c|c|c|c|c|}
		\hline
		$p \backslash r$& 1&2&3&4&5&6\\
		\hline
		1&16&8&4&2&1&1\\\hline
		2&46&32&22&15&15&16\\\hline
		3&46&38&31&31&38&46\\\hline
		4&16&15&15&22&32&46\\\hline
		5&1&1&2&4&8&16\\\hline
	\end{tabular}
\hfil 
	\begin{tabular}{|c||c|c|c|c|c|}
		\hline
		$p \backslash r$& 1&2&3&4&5\\
		\hline
		1 & 8 & 4 & 2 & 1 & 1 \\\hline
		2 & 14 & 10 & 7 & 7 & 8 \\\hline
		3 & 8 & 7 & 7 & 10 & 14 \\\hline
		4 & 1 & 1 & 2 & 4 & 8 \\\hline
	\end{tabular}

\caption{The number of $(r,p)$-toppleable permutations for $n=5$ on the left and $n=4$ on the right.}
\label{tab:n4 and 5rp}
\end{table}  
\end{center}

Poly-Bernoulli numbers of type B arise in the enumeration of toppleable permutations also as the number of $(1,p)$-toppleable permutations.

\begin{proposition}
\label{prop:1p_top_perm}
  
	We have $|\mathcal{T}_n^{(1,n)}| = 1$ and
	\begin{align*}|\mathcal{T}_n^{(1,p)}| = \sum_{r=1}^{n} |\mathcal{T}_{n-1}^{(r,p)}|, \quad 1 \leq p \leq n-1.
	\end{align*}

\end{proposition}

\begin{proof}
It is easy to see that if $p = n$ and $r =1$, there is a single permutation which topples to the identity. Now, suppose $p < n$.
	Let $\pi^{(1,p)}\in \mathcal{T}_n^{1,p}$. Then $1$ is at site $p$ together with another chip. Let $a$ be this chip, $2\leq a \leq n+1$. We define the \emph{left pass} as the consecutive topplings in a pass on the left hand side of $p$. Using our previous notation for unlabeled configurations, we have

\begin{align*}
(\_,1,\ldots, 1,2,\hat{\_},2,1,\ldots, 1,\_)
\rightarrow
(1,\_,1,\ldots,1, 1,1,1,\hat{1},2,1,\ldots, 1,\_),
\end{align*}
after a left pass.
After the first left pass the chip $1$ is frozen on the leftmost site, and the chip $a$ moves one site to the right, so to the $(p+1)$'th site. Thus, we have a configuration where the first site (zeroth) is occupied by the chip $1$, the next site (first) is empty, and there are two chips at the $(p+1)$'th site (of which one is the chip $a$).  Ignoring the site $0$ with the chip $1$, and reducing all other chip labels by one, we obtain a configuration on $n$ chips with two chips at the $p$'th site (note that by ignoring the site $0$ the sites are shifted to the left). If our initial configuration was $(1,p)$-toppleable, this configuration has to be $(a-1,p)$-toppleable. Conversely, given a $(a-1,p)$-toppleable permutation in $S_{n-1}$, we can obtain a $(a,p)$-toppleable permutation in $S_n$ by reversing the process above.
Since $1\leq a-1\leq n$, the sum goes from $1$ to $n$. 
\end{proof}

We can now count the number of $(1,p)$-toppleable permutations.

\begin{corollary}
	The number of $(1,p)$-toppleable permutations in $S_n$ is  $B_{n-p,p}$.
\end{corollary}

\begin{proof}
By \cref{prop:1p_top_perm}, we have to sum over the number of $(r,p)$-toppleable permutations in $S_{n-1}$ for all possible $r$. But this is the same as twice the number of $p$-toppleable configurations in $\mathcal{S}(n-1,p)$ because every such configuration is represented twice. \cref{thm:p-resultant-configs} then proves the result.
\end{proof}

As an illustration of this result, compare the first column of the tables in \cref{tab:n4 and 5rp} with the entries in \cref{tab:data-polybern}(a).
 The characterisation in \cref{prop: perm_char} gives us the class of Vesztergombi permutations and the class of Callan permutations that corresponds to $(r,p)$-toppleable permutations for a given $r$.
 
Using these correspondences we can generalize \cref{prop:1p_top_perm} and give a similar recursion for arbitrary $r$.

\begin{theorem}
\label{thm: general-rec}

\[
|\mathcal{T}_n^{(r,p)}| =
\begin{cases}
\ds \sum_{i=r}^{n} 	|\mathcal{T}_{n-1}^{(i,p)}| & r\leq n-p+1, \\
\ds \sum_{i=1}^{r-1} |\mathcal{T}_{n-1}^{(i,p-1)}| & r > n-p+1.
\end{cases}
\]
\end{theorem}

\begin{proof}
We know from \cref{prop: perm_char} that $|\mathcal{T}_n^{r,p}|$ is the number of $(n-p+1,p)$-Callan permutations starting with $r$. We need to consider two cases separately: $r\leq n-p+1$ and $r>n-p+1$. 
	If $r\leq n-p+1$, i.e.,  $r$ is an underlined element and by definition this means that $r$ is followed by a greater underlined element or an overlined element. If we delete $r$ and reduce the value of each element greater than $r$ by one, we obtain a Callan permutation with $n-p$ underlined elements and $p$ overlined elements. The number of such permutations is the sum of $|\mathcal{T}_{n-1}^{(i,p)}| $ where $i$ goes from $r$ to $n$. Similarly, if $r>n-p+1$, which means that $r$ is an overlined element, then $r$ is followed by a smaller overlined element or an underlined element. After deleting $r$ and reducing the remaining elements greater than $r$ by one, we obtain a Callan permutation with $n-p+1$ underlined and $p-1$ overlined elements and a starting element smaller than $r$. The number of such permutations is the sum $|\mathcal{T}_{n-1}^{(i,p-1)}|$ where $i$ goes from $1$ to $r-1$.
\end{proof}

We are now in a position to enumerate $(r,p)$-toppleable permutations.

\begin{proof}[Proof of \cref{thm:r-p-top-perm}]
By \cref{prop: perm_char}, it suffices to look at $(n-p+1,p)$-Callan permutations starting with $r$. 
The set of $(n-p+1,p)$-Callan permutations is enumerated by 
$B_{n-p+1,p}$. As we showed in \cref{prop:1p_top_perm}, 
the number of such permutations for $r=1$ is $B_{n-p,p}$, which is the same as $\Delta^{0} \big( B_{n-p,p} \big)$. Throughout, 
$\Delta$ acts on the first index.

For $r=2$, we have to consider Callan permutations starting with the underlined element $r=2$. Delete the starting element $2$. What we get is a permutation on the set of underlined elements $\{1,3,4,\ldots,n-p+1\}$ and overlined elements $\{n-p+2,\ldots, n+1\}$ starting with an underlined element greater than $2$, or with an overlined element. If we reduce the elements greater than $2$ by one, we get a permutation with $n-p$ underlined elements and $p$ overlined elements \textit{not} starting with $1$. Thus, we see that the number of Callan permutations of $n+1$ elements starting with $2$ is the difference of the number of Callan permutations of $n$ elements and Callan permutations of $n$ elements starting with $1$, namely
 $ B_{n-p,p}-B_{n-p-1,p}$, which is $\Delta^1 \big( B_{n-p-1,p} \big)$.

In general we can argue the same way. Let $r\leq n-p+1$ as in the condition of the theorem which means that $r$ is an underlined element in the corresponding Callan permutation. Now, given a Callan permutation starting with the underlined element $r$, if we delete the starting element, we obtain Callan permutation of the underlined elements $\{1,2,\ldots, r-1,r+1,\ldots, n-p+1\}$ and overlined elements $\{n-p+2,\ldots, n+1\}$ starting with an underlined element that is greater than $r$ or with an overlined element. If we reduce the elements greater than $r$ by one, we obtain a Callan permutation of underlined elements $\{1,\ldots, n-p\}$ and overlined elements $\{n-p+1,\ldots, n\}$ \textit{not} starting with any of the following elements: $1$, $2$, \ldots, $r-1$. Denoting by $f_{n-p,p}(r)$ the number of $(n-p,p)$-Callan permutations starting with $r$ we have shown that
\begin{align*}
	f_{n+1-p,p}(r) = B_{n-p,p} - f_{n-p,p}(1) - \cdots -f_{n-p,p}(r-1).
\end{align*}

The remainder of the proof follows by induction on $n$. The result is easily verified for small values of $n$ by explicit computation. According to the induction hypothesis 
\[
f_{n-p,p}(k) = \Delta^{k-1}(B_{n-p-k,p}) 
= \sum_{j=0}^{k-1} (-1)^j \binom{k-1}{j} B_{n-p-j-1,p}.
\]
By the above equation,
\[
f_{n+1-p,p}(r) = B_{n-p,p} - \sum_{k=1}^{r-1} f_{n-p,p}(k),
\]
and plugging in the induction assumption, we obtain
\begin{align*}
f_{n+1-p,p}(r) =& B_{n-p,p} - \sum_{k=1}^{r-1} \sum_{j=0}^{k-1} (-1)^j \binom{k-1}{j} B_{n-p-j-1,p} \\
=& B_{n-p,p} - \sum_{j=0}^{r-2} (-1)^j B_{n-p-j-1,p} \sum_{k=j+1}^{r-1} \binom{k-1}{j}.
\end{align*}
The inner sum on $k$ now gives us $\binom{r-1}{j+1}$ and we end up with
\[
f_{n+1-p,p}(r) = B_{n-p,p} - \sum_{j=0}^{r-2} (-1)^j \binom{r-1}{j+1} B_{n-p-j-1,p},
\]
and the right hand side is exactly $\Delta^{r-1}(B_{n-p+1-r,p})$, completing the proof.
\end{proof}

We now recover the result of \cite{AHT}.

\begin{proof}[Proof of \cref{cor:aht-theorem}]
	The corollary follows from the fact that the two types of poly-Bernoulli numbers are inversion transforms of each other \eqref{relationB_C} and by \cref{thm:r-p-top-perm}. However, it is easy to prove the formula directly using the bijection with Callan permutations.
	
	By \cref{prop: perm_char}, $(p,p)$-toppleable permutations are in one-to-one correspondence with ($p$, $n+1-p$)-Callan permutations starting with the greatest underlined element $r=p$.
	(Ignoring the first element, we obtain a Callan permutation starting with an overlined element, since subsequences of underlined elements are increasing.) 
	By \cref{thm:polyC-objects}(4), we obtain the result.
	
	If $r=p+1$, the corresponding Callan permutations start with the smallest overlined element. So, ignoring this first element, we obtain a Callan permutation that start with an underlined element, since subsequences of overlined elements are ordered decreasingly. By symmetry between underlined and overlined elements in Callan permutations, the number of Callan permutations starting with an underlined element are also enumerated by the same poly-Bernoulli number. 
\end{proof}

We have another formula for $|\mathcal T_{n}^{(r,p)}|$ using the poly-Bernoulli numbers of type C.

\begin{corollary}
\label{cor:Tn-formula}
Let $n$, $p\leq n$ be integers. We have
\begin{align*}
|\mathcal T_{n}^{(r,p)}|=
\begin{cases}
\ds \sum_{i=0}^{n-p+1-r}\binom{n-p+1-r}{i}C_{p,n-p-i} & r \leq n-p+1, \\[0.5cm]
\ds \sum_{i=0}^{r-n+p-2}\binom{r-n+p-2}{i}C_{n-p+1,p-i-1} & r > n-p+1 .
\end{cases}
\end{align*}
\end{corollary}

\begin{proof}
	Let $r\leq n-p+1$. By \cref{prop: perm_char}, $|\mathcal T_{n}^{(r,p)}|$ is the number of $(n-p+1,p)$-Callan permutations starting with an underlined element $r$. Hence, the first block is underlined. Denoting by $i$ the number of underlined elements in this first block besides $r$, we can construct this block in $\binom{n-p+1-r}{i}$ ways, since the underlined elements are arranged increasingly. 
Ignoring this first block, we have a Callan permutation with $n-p+i$ underlined and $p$ overlined elements starting with an overlined element. This is known to be enumerated by 
$C_{p,n-p-i}$ by \cref{thm:polyC-objects}(4). 

Similarly, if $r>n-p+1$ the starting block contains overlined elements and in this case we have $\binom{r-(n-p+1)-1}{i}$ possibilities to choose the elements into this block. The remaining, $n-p+1$ underlined and $p-1-i$ overlined elements construct a Callan permutation that starts with an overlined element, which is known to be enumerated by $C_{n-p+1,p-i-1}$. Note that we used in this argument the symmetry property of $C_{n,k}$.
\end{proof}

Another consequence of our bijection is the following relation that can be seen from the data in \cref{tab:n4 and 5rp}.

\begin{corollary}
Let $n$, $p$, and $r$ be integers, such that $p\leq n$ and $r\leq n-p+1$. Then 
\begin{align*}
\sum_{r=1}^{n-p+1} |\mathcal{T}_n^{(r,p)}| = C_{n-p+1,p}.
\end{align*}
Similarly, if $r>n-p+1$, then
\begin{align*}
\sum_{r=n-p+2}^n |\mathcal{T}_n^{(r,p)}| = C_{p,n-p+1}.
\end{align*}
\end{corollary}

\begin{proof}
$\sum_{r=1}^{n-p+1} |\mathcal{T}_n^{(r,p)}|$ counts all $(r,p)-$toppleable permutations where $r\leq n-p+1$. According to the bijection this is the same as the number of $(n-p+1,p)-$ Callan permutations starting with an underlined element, which is given by 
$C_{n-p+1,p}$. The second statement is proven analogously.
\end{proof}

\begin{proof}[Proof of \cref{thm:top-perms-fixed-p}]
By \cref{prop: perm_char}, any toppleable configuration $C \in \mathcal{S}(n,p)$ satisfies
\[
p + t - n - 1 \leq C^{-1}_t \leq p + t - 1, \quad 1 \leq t \leq n.
\]
Suppose $\pi \in \mathcal{T}_n^{(r,p)}$ for all $1 \leq r \leq n+1$. Then the configuration $C_r$ obtained by adding $r$ at site $p$ to $\pi$ and shifting values larger than $r$ by $1$ must satisfy these inequalities for each $r$. 
Now, every value $t \in [n]$ appears both as $t$ and as $t+1$ as $r$ varies. Therefore, $\pi$ has to satisfy both
\[
p + t - n - 1 \leq C^{-1}_t \leq p + t - 1 \text{ and }
p + t - n \leq C^{-1}_t \leq p + t,
\]
for each $t$, proving the first statement. The second then follows from \cref{thm:polyC-objects}(2).
\end{proof}

\section{Resultant permutations}
\label{sec:finalconfigs}

We now classify $p$-resultant permutations in $S_n$. 

\begin{proof}[Proof of \cref{thm:p-resultant-perms}]
By \cref{prop:symm}, it suffices to take $p \leq \floor{n/2}$.
We want to show that every $p$-resultant permutation in $S_n$ can be written as a concatenation of a permutation of the elements $1,2,\ldots, n-p$ followed by a permutation of $n-p+1,\ldots, n$. (Note that the initial configuration has $n-1$ sites, so $C\in S(n-1,p)$.)
Let $\pi$ be a $p$-resultant permutation.
We recall some facts from \cref{sec:toppleable}. 
	The first and last sites are frozen after the first pass; hence, the first entry, $\pi_1$ is the smallest element of the set $S_0$. Similarly, the last entry, $\pi_n$ is the greatest element of the set $T_0$. 
	$S_1$ is the set of chips in the left part and on the site $p$ after the first pass. It is possible that one chip from the $p$'th site, say $a$, topples out of the set $S_0$ and another chip, say $b$, topples into it from the right. In this case we have $S_1 = S_0\setminus\{\pi_1,a\}\cup\{b\}$. Alternatively, the chip that topples in the first topple to the right from the $p$'th site may topple back to the left in the next topple.
In this case we have $S_1 = S_0\setminus\{\pi_1\}$. (Note that the fixed entry $\pi_1$ is not contained in the active part of a configuration, and so is not included in $S_1$.) In both cases we have $|S_1|=|S_0|-1$.

In general, $S_i$ is defined as the set of chips in the left part and on the site $p$ after the $i$'th pass.
Note that after the $i$'th pass the chip $\min(S_i)$ is fixed at the $i$'th site and, arguing as above, after each pass the size of the left part is reduced by one, $|S_i|=|S_{i-1}|-1$.
Since the size of $S_0$ is $p+1$, it contains at least one element from the set $V=\{1,2, \ldots, n-p\}$. Hence, the minimum of $S_0$ (which is $\pi_1$) is from the set $V=\{1,2, \ldots, n-p\}$. The size of $S_1$ is $p$, but since the greatest element of $T_0$ is fixed as $\pi_n$ after the first pass, at least one element from the set $V\setminus\{\pi_1\}$ is contained in $S_1$. Hence, the minimum of $S_1$, $\pi_2$ is from the set $V\setminus\{\pi_1\}$.  In general, $S_i$ has to contain at least one element from $V\setminus\{\pi_1, \pi_2,\ldots, \pi_{i-1}\}$. Hence, $\pi_{i+1}=\min(S_{i})$ is from the set  $V\setminus\{\pi_1, \pi_2,\ldots, \pi_{i-1}\}$.

We know that there are $p$ passes in a toppling process (we assumed $p \leq \floor{n/2}$). By symmetry this implies that $\pi_{n+1-p},\pi_{n+2-p},\ldots, \pi_{n}$ are elements from the set $\{n+1-p, \ldots, n\}$, and completes the proof.
\end{proof}

We now prove \cref{lem: topp_lrm}. Let $\pi$ be the $p$-resultant permutation that is the concatenation of the permutation $\pi^{\text{left}}$ of $S_{n-p}$ and $\pi^{\text{right}}$ of $S_p$.

\begin{proof}[Proof of \cref{lem: topp_lrm}]

It suffices to show that in a configuration that topples to $\pi$, the relative order of elements that are not left-to-right maxima in $\pi^{\text{left}}$ is the same as in $\pi^{\text{left}}$ itself. 
We will show that these elements have a position in the initial configuration to the right of position $p$. Similarly, the relative order of elements that are not right-to-left minima in $\pi^{\text{right}}$ is the same as in $\pi^{\text{right}}$ itself. These elements have a position in the initial configuration to the left of position $p$. By symmetry it is enough to show the first part of the statement. 

We will continue to use the notation from the proof of \cref{thm:p-resultant-perms} above.

We know that $\pi_1=\min S_0, \pi_2 =\min S_1$ and so on. If $\pi_2<\pi_1$ (so that $\pi_2$ is not a left-to-right maximum), then $\pi_2\in S_1\setminus S_0$. Note also that since $S_1\setminus S_0$ contains only at most one element, we have actually $\{\pi_2\}=S_1\setminus S_0$. Moreover, $\pi_2$ has to be at the $(p+1)$'th before the first pass. The same argument can be used to show that if both $\pi_2,\pi_3<\pi_1$, then $\{\pi_3\} = S_2\setminus S_1 $ and that $\pi_3$ was at the position $(p+1)$ before the second pass and position $(p+2)$ before the first pass.

In general, let $k$ be the minimal index such that $\pi_{k}>\pi_1$ but $\pi_{i}<\pi_1$ for all $1 < i < k$. Then $\{\pi_{i+1}\} = S_{i}\setminus S_{i-1}$ and $\pi_{i+1}$ is in the $(p+1)$'th site before the $i$'th pass. 

Hence, the elements $\pi_2,\ldots, \pi_{k-1}$ had to be in the initial configuration on the sites $p+1,\ldots, p+{k-3}$ in this order. Note that $\pi_k$ could be contained either in $S_0$ or only in $S_{k-1}$. 
For the elements between $\pi_k$ and the next greater element, we  argue the same way starting from $S_k$ and the $k$'th pass. 
Continuing this way completes the proof.
\end{proof}

We can now enumerate the number of configurations toppling to a given resultant permutation.

\begin{proof}[Proof of \cref{thm:p-resultant-configs}]
Let $\pi \in S_{n+1}$ be a $p$-resultant permutation with $i$ left-to-right maxima in $\pi^{\text{left}}$ and $j$ right-to-left minima in $\pi^{\text{right}}$. 
We will write $\pi = \pi^{\text{left}} \oplus \pi^{\text{right}}$.
By \cref{lem: topp_lrm}, the relative order of elements which are not left-to-right maxima/right-to-left minima is fixed, and they do not influence the number of possible initial configurations that topple to $\pi$. 

Let $\pi_1=\pi_{\ell_1}<\pi_{\ell_2}<\cdots <\pi_{\ell_i}$ be the left-to-right maxima in $\pi^{\text{left}}$ and $\pi_{r_1}<\pi_{r_2}<\cdots< \pi_{r_{j}}=\pi_{n+1}$ the right-to-left minima in $\pi^{\text{right}}$.
We now define a bijection, $\phi$,  between the set of configurations $C_{\pi}$ that topple to $\pi$ 
with left-to-right maxima $\pi_1=\pi_{\ell_1}<\pi_{\ell_2}<\cdots <\pi_{\ell_i}$ left of $p$ and  right-to-left minima $\pi_{r_1}<\pi_{r_2}<\cdots< \pi_{r_{j}}=\pi_{n+1}$ right of $p$
and the set of $j$-toppleable configurations $C\in \mathcal{S}(i+j-1,j)$ as follows. 

Given a configuration $C_{\pi}\in \mathcal{S}(n,p)$ that topple to the permutation $\pi=\pi^{\text{left}} \oplus \pi^{\text{right}}$ construct a configuration on $i+j$ having two chips on the site $j$ as follows:
\begin{itemize}
	\item[1.] Delete all the chips (and their sites) that are neither left-to-right maxima of $\pi^{\text{left}}$ nor right-to left minima of $\pi^{\text{right}}$.
	\item[2.] Relabel naturally the left-to-right maxima $\phi(\pi_{\ell_k}) = k $  and the right-to-left minima $\phi(\pi_{r_{j}})=i+j$. 
\end{itemize} 
Notice that $\phi$ acting on the resultant configuration gives the identity permutation in $S_{i+j}$.

We prove that $\phi$ is a bijection by explicitly constructing its inverse as follows. Given a $j$-toppleable configuration $C\in \mathcal{S}(i+j-1,j)$ and a permutation $\pi=\pi^{\text{left}} \oplus \pi^{\text{right}}$ such that $\pi^{\text{left}}$ is a permutation of $[n-p+1]$ with $i$ left-to-right maxima, $\pi_1=\pi_{\ell_1}<\pi_{\ell_2}<\cdots <\pi_{\ell_i}$, and $\pi^{\text{right}}$ with $j$ right-to-left minima,  $\pi_{r_1}<\pi_{r_2}<\cdots <\pi_{r_{j}}=\pi_{n+1}$, construct a configuration $C_{\pi}$ as follows:
\begin{itemize}
	\item[1.] Relabel the chips $k$ with $1\leq k\leq i$ by $\pi_{\ell_k}$ and the chips $k$ with $i<k\leq i+j$ by $\pi_{r_{k-i}}$. 
	\item[2.] For $1 \leq k \leq i-1$, insert between the $k$'th and $(k+1)$'th sites in $C$ the chips 	$\pi_{\ell_k+1}$, $\pi_{\ell_k+2}$, $\ldots$, $\pi_{\ell_{k+1}-1}$ in the same relative order. 
Extend the same way the original configurations by the non-right-to-left minimas of $\pi^{\text{right}}$ on the left hand side of the $j$'th site. Similarly, for $2 \leq k \leq j$, insert between the $(k-1)$'th and $k'$th sites the chips $\pi_{r_{k-1}-1}$, $\pi_{r_{k-1}-2}$, $\ldots$, $ \pi_{r_{k}+1}$ in the same relative order. 
\end{itemize}  
 From the proof of \cref{lem: topp_lrm} it follows that $\phi$ is a bijection.  
Now, the number of $j$-toppleable configurations in $\mathcal{S}(i+j-1,j)$ is $B_{i,j}/2$ by \cref{thm: top-pB}.
\end{proof}

\begin{proof}[Proof of \cref{thm:p-resultant-perms-r}]

We will prove the first part of the claim. The second then follows by the symmetry in \cref{prop:symm}.
So, suppose $r\leq n-p$.
From \cref{thm:p-resultant-perms}, we know that $r \in \pi^{\text{left}}$. It remains to prove that $r$ is a left-to-right maximum therein.

Since $r$ is on the $p$'th site, $r\in S_0$. If $r=\min S_0=\pi_1$, we are done, since $\pi_1$ is by definition a left-to-right maximum. If $r \neq \pi_1$ there is an index $1<i\leq n-p$ such that  $r=\pi_i$. We need to show that all the entries $\pi_1,\pi_2,\ldots,\pi_{i-1} $ are smaller than $r$.
	
Let $j^*$ be the least index such that $r\in S_0,S_1,\ldots, S_{j^*-1}$ but $r \notin S_{j^*}$, i.e., $r$ topples in the $j^*$'th pass out of the left part. Then, $\pi_1=\min S_0, \ldots, \pi_{j^*}=\min S_{j^*-1}$ are all smaller than $r$.
Before the $j^*$'th pass, $r$ is at the $p$'th site with a smaller chip and on the $(p+1)$'th site there is also a smaller chip than $r$, say $d$. So we have $S_{j^*}=S_{j^*-1}\setminus \{\pi_{j^*},r\}\cup \{d\},$ with $d<r$. This implies $\pi_{j^*+1}=\min S_{j^*}\leq d<r$. 
	
Now let $i^*$ be the first index such that $S_{j^*}$, $S_{j^*-1}$, $\ldots$, $S_{i^*-1}$ do not contain the chip $r$, but $r\in S_{i^*}$, i.e., the chip $r$ topples back to the left part in the $i^*$'th pass.
Note that each chip in $S_j\setminus S_{j-1}$ 
with $j^*<j<i^*$ joins at least once the site with $r$ and, hence there are necessarily smaller than $r$ (otherwise $r$ would topple to the left by the toppling rule). This again implies that $\min S_{j^*}=\pi_{j^*+1}$, $\min S_{j^*+1}=\pi_{j^*+2}$, $\ldots$, $\min S_{i^*-1}=\pi_{i^*}$ are all smaller than $r$.
If $\min S_{i^*}=r$, we are done. If not we can repeat the above argument with analogous definitions of $j^{**}$ (as the first pass after $i^*$'th when $r$ topples out of the left part) and $i^{**}$ (as the first pass when $r$ topples back to the left part). Continuing this way, we argue that $r$ has to be a left-to-right maximum in $\pi^{\text{left}}$, completing the proof.
\end{proof}

\section*{Acknowledgements}

The first author (AA) was partially supported by the UGC Centre for Advanced Studies and by Department of Science and Technology grant EMR/2016/006624. 

We thank the anonymous referee for comments. The second author would
also like to thank Toshiki Matsusaka for helpful comments about
poly-Bernoulli numbers. 

\bibliographystyle{alpha}
\bibliography{bibliography}

\begin{thebibliography}{GHMP21}

\bibitem[AHT20]{AHT}
Arvind Ayyer, Daniel Hathcock, and Prasad Tetali.
\newblock Toppleable permutations, excedances and acyclic orientations, 2020.

\bibitem[AK99]{ArakawaKaneko}
Tsuneo Arakawa and Masanobu Kaneko.
\newblock Multiple zeta values, poly-{B}ernoulli numbers, and related zeta
  functions.
\newblock {\em Nagoya Math. J.}, 153:189--209, 1999.

\bibitem[BH15]{BH}
Be\'{a}ta B\'{e}nyi and P\'{e}ter Hajnal.
\newblock Combinatorics of poly-{B}ernoulli numbers.
\newblock {\em Studia Sci. Math. Hungar.}, 52(4):537--558, 2015.

\bibitem[BH17]{benyi-hajnal-2017}
Be\'{a}ta B\'{e}nyi and P\'{e}ter Hajnal.
\newblock Combinatorial properties of poly-{B}ernoulli relatives.
\newblock {\em Integers}, 17:Paper No. A31, 26, 2017.

\bibitem[Bre08]{Brewbaker}
Chad Brewbaker.
\newblock A combinatorial interpretation of the poly-{B}ernoulli numbers and
  two {F}ermat analogues.
\newblock {\em Integers}, 8:A02, 9, 2008.

\bibitem[CGS14]{cameron2014acyclic}
P.~J. Cameron, C.~A. Glass, and R.~U. Schumacher.
\newblock Acyclic orientations and poly-bernoulli numbers.
\newblock {\em arXiv preprint arXiv:1412.3685}, 2014.

\bibitem[dALN15]{Lundberg}
Rodrigo~Ferraz de~Andrade, Erik Lundberg, and Brendan Nagle.
\newblock Asymptotics of the extremal excedance set statistic.
\newblock {\em European J. Combin.}, 46:75--88, 2015.

\bibitem[FK21]{klivans-felzenszwalb-2021}
Pedro Felzenszwalb and Caroline Klivans.
\newblock Flow-firing processes.
\newblock {\em J. Combin. Theory Ser. A}, 177:105308, 18, 2021.

\bibitem[GHMP19]{galashin-et-al-2019}
Pavel Galashin, Sam Hopkins, Thomas McConville, and Alexander Postnikov.
\newblock Root system chip-firing {I}: interval-firing.
\newblock {\em Math. Z.}, 292(3-4):1337--1385, 2019.

\bibitem[GHMP21]{Galashin2021}
Pavel Galashin, Sam Hopkins, Thomas McConville, and Alexander Postnikov.
\newblock Root system chip-firing ii: Central firing.
\newblock 13:10037--10072, 2021.

\bibitem[GKP94]{knuth-graham-patashnik-1994}
Ronald~L. Graham, Donald~E. Knuth, and Oren Patashnik.
\newblock {\em Concrete mathematics}.
\newblock Addison-Wesley Publishing Company, Reading, MA, second edition, 1994.
\newblock A foundation for computer science.

\bibitem[HMP17]{HMP}
Sam Hopkins, Thomas McConville, and James Propp.
\newblock Sorting via chip-firing.
\newblock {\em Electron. J. Combin.}, 24(3):Paper No. 3.13, 20, 2017.

\bibitem[HP19]{hopkins-postnikov-2019}
Sam Hopkins and Alexander Postnikov.
\newblock A positive formula for the {E}hrhart-like polynomials from root
  system chip-firing.
\newblock {\em Algebr. Comb.}, 2(6):1159--1196, 2019.

\bibitem[Kan97]{KanekoIntro}
Masanobu Kaneko.
\newblock Poly-{B}ernoulli numbers.
\newblock {\em J. Th\'{e}or. Nombres Bordeaux}, 9(1):221--228, 1997.

\bibitem[Kin06]{king-2006}
Andrew King.
\newblock Generating indecomposable permutations.
\newblock {\em Discrete Math.}, 306(5):508--518, 2006.

\bibitem[KKL13]{KimKrotovLee}
Hyun~Kwang Kim, Denis~S. Krotov, and Joon~Yop Lee.
\newblock Matrices uniquely determined by their lonesums.
\newblock {\em Linear Algebra Appl.}, 438(7):3107--3123, 2013.

\bibitem[KL20]{klivans-liscio-2020+}
Caroline Klivans and Patrick Liscio.
\newblock Confluence in labeled chip-firing.
\newblock {\em arXiv preprint arXiv:2006.12324}, 2020.

\bibitem[Kla03]{klazar-2003}
Martin Klazar.
\newblock Irreducible and connected permutations.
\newblock {\em Institut teoretick{\'e} informatiky (ITI) Series}, (122), 2003.

\bibitem[KLM21]{Lundberg_Khera}
Jessica Khera, Erik Lundberg, and Stephen Melczer.
\newblock Asymptotic enumeration of lonesum matrices.
\newblock {\em Adv. in Appl. Math.}, 123:102118, 17, 2021.

\bibitem[Knu73]{knuth-taocp3}
Donald~E. Knuth.
\newblock {\em The art of computer programming. {V}olume 3}.
\newblock Addison-Wesley Publishing Co., Reading, Mass.-London-Don Mills, Ont.,
  1973.
\newblock Sorting and searching, Addison-Wesley Series in Computer Science and
  Information Processing.

\bibitem[Lau07]{Launois}
St\'{e}phane Launois.
\newblock Combinatorics of {$\mathscr H$}-primes in quantum matrices.
\newblock {\em J. Algebra}, 309(1):139--167, 2007.

\bibitem[LV78]{lovasz-vesztergombi-1978}
L\'{a}szl\'{o} Lov\'{a}sz and Katalin Vesztergombi.
\newblock Restricted permutations and {S}tirling numbers.
\newblock In {\em Combinatorics ({P}roc. {F}ifth {H}ungarian {C}olloq.,
  {K}eszthely, 1976), {V}ol. {II}}, volume~18 of {\em Colloq. Math. Soc.
  J\'{a}nos Bolyai}, pages 731--738. North-Holland, Amsterdam-New York, 1978.

\bibitem[OEI20]{OEIS}
OEIS.
\newblock The {O}n-{L}ine {E}ncyclopedia of {I}nteger {S}equences.
\newblock Published electronically at \url{http://oeis.org}, 2020.

\bibitem[Sjo07]{Sjostrand}
Jonas Sjostrand.
\newblock Bruhat intervals as rooks on skew {F}errers boards.
\newblock {\em J. Combin. Theory Ser. A}, 114(7):1182--1198, 2007.

\bibitem[Ves74]{vesztergombi-1974}
Katalin Vesztergombi.
\newblock Permutations with restriction of middle strength.
\newblock {\em Studia Sci. Math. Hungar.}, 9:181--185 (1975), 1974.

\end{thebibliography}

\end{document}